\documentclass[oneside,a4paper,12pt,notitlepage]{article}
\usepackage[left=0.8in, right=0.8in, top=1.5in, bottom=1.5in]{geometry}
\usepackage[T1]{fontenc} 
\usepackage[utf8]{inputenc} 
\usepackage[english]{babel} 
\usepackage{lipsum} 
\usepackage{lmodern}
\usepackage{amssymb}
\usepackage{amsthm}
\usepackage{bm}
\usepackage{mathtools}
\usepackage{xcolor}
\usepackage{dsfont}

\usepackage{braket}
\usepackage{esint}
\newcommand{\abs}[1]{{\left|#1\right|}}
\newcommand{\norma}[1]{{\left\Vert#1\right\Vert}}

\usepackage{booktabs}
\usepackage{graphicx}
\usepackage{tikz}
\usetikzlibrary{patterns}
\usepackage{multicol}
\usepackage{caption}
\usepackage{enumerate}
\usepackage[skins,theorems]{tcolorbox}
\tcbset{highlight math style={enhanced,
		colframe=black,colback=white,arc=0pt,boxrule=1pt}}
\def\XXint#1#2#3{{\setbox0=\hbox{$#1{#2#3}{\int}$}
  \vcenter{\hbox{$#2#3$}}\kern-.5\wd0}}

\theoremstyle{definition}
\newtheorem{definizione}{Definition}[section]
\theoremstyle{plain}
\newtheorem{teorema}{Theorem}[section]

\newtheorem{lemma}[teorema]{Lemma}
\newtheorem{prop}[teorema]{Proposition}

\theoremstyle{definition}
\newtheorem{esempio}{Example}[section]
\newtheorem{oss}[esempio]{Remark}

\DeclareMathOperator{\Supp}{\text{Supp}}

\DeclareMathOperator{\R}{\mathbb{R}}

\DeclareMathOperator{\diam}{\text{diam}}

\usepackage{hyperref}
\hypersetup{linktoc=none, bookmarksnumbered, colorlinks=true, linkcolor=red}
\title{On the solutions to $p$-Poisson equation with Robin boundary conditions when $p$ goes to $+\infty$}
\author{Vincenzo Amato, Alba Lia Masiello, Carlo Nitsch, Cristina Trombetti}
\date{\today}
\newcommand{\Addresses}{{
 \bigskip 
 \footnotesize 
 
 \textsc{Dipartimento di Matematica e Applicazioni ``R. Caccioppoli'', Universit\`a degli studi di Napoli Federico II, Via Cintia, Complesso Universitario Monte S. Angelo, 80126 Napoli, Italy.}\par\nopagebreak 
 
 \medskip 
 
 \textit{E-mail address}, V.~Amato: \texttt{vincenzo.amato@unina.it} 
 
  \medskip 
 
 \textit{E-mail address}, A.L.~Masiello: \texttt{albalia.masiello@unina.it} 
  
 \medskip 
 
 \textit{E-mail address}, C.~Nitsch: \texttt{c.nitsch@unina.it}

  \medskip 
 
 \textit{E-mail address}, C.~Trombetti: \texttt{cristina@unina.it} 

}}     

\begin{document}
	\maketitle
	 \vspace{-0.8cm}
	\begin{abstract}
		 We study the behaviour, when $p \to +\infty$, of the first $p$-Laplacian eigenvalues with Robin boundary conditions and the limit of the associated eigenfunctions. We prove that the limit of the eigenfunctions is a viscosity solution to an eigenvalue problem for the so-called $\infty$-Laplacian.
		
		 Moreover, in the second part of the paper, we focus our attention on the $p$-Poisson equation when the datum $f$ belongs to $L^\infty(\Omega)$ and we study the behaviour of solutions when $p\to\infty$. 
		 
	\end{abstract}

 \section{Introduction}
Let $\beta$ be a positive parameter and let $\Omega$ be a bounded and open set of $\R^n$, $n\ge 2$, with Lipschitz boundary. 

We study the $\infty$-Laplacian eigenvalue problem with Robin
boundary conditions

\begin{equation}
\label{ciao2}
  \begin{cases}
  \min\Set{\abs{\nabla u} - \Lambda u, - \Delta_\infty u}=0 &\text{ in } \Omega,\\
  -\min\Set{\abs{\nabla u} - \beta u, -\displaystyle{\frac{\partial u}{\partial \nu}} }=0 &\text{ on } \partial \Omega,
\end{cases}
\end{equation}

where $\Delta_\infty$, the so-called $\infty$-Laplacian, is defined by
$$\Delta_\infty u = \left\langle D^2 u \cdot \nabla u, \nabla u \right\rangle.$$


We refer to problem \eqref{ciao2} as the $\infty$-Laplacian eigenvalue problem because it can be seen as limit, in some sense, of the $p$-Laplacian eigenvalues problem
\begin{equation}
\label{autoval}
\begin{cases}
-\Delta_p u= \Lambda_p \abs{u}^{p-2} u & \text{ in } \Omega \\
\abs{\nabla u}^{p-2} \displaystyle{\frac{\partial u}{\partial \nu}} + \beta^p \abs{u}^{p-2}u =0 & \text{ on } \partial \Omega.
\end{cases}
\end{equation}

A function $u_p\in W^{1,p}(\Omega)$ is a weak solution to \eqref{autoval} if it satisfies

\begin{equation*}
  \int _\Omega \abs{\nabla u_p}^{p-2} \nabla u_p \nabla \varphi \, dx+ \beta^p \int_{\partial\Omega}\abs{u_p}^{p-2}u_p\varphi \, d \mathcal{H}^{n-1} = \Lambda_p \int_{\Omega} \abs{u_p}^{p-2} u_p \varphi \, dx, \quad \forall \varphi \in W^{1,p}(\Omega),
\end{equation*}
where $\beta$ and $\Lambda_p$ are both positive.

It is well known that the first eigenvalue of the $p$-Laplacian is the minimum of the following Rayleigh quotient
\begin{equation}
  \label{Rel}
  \Lambda_p =\inf_{\substack{w\in W^{1,p}(\Omega) \\ \norma{w}_{L^p(\Omega)}=1}} \left\lbrace \int_\Omega \abs{\nabla w}^p \, dx +\beta^p \int_{\partial\Omega} \abs{w}^p \, d\mathcal{H}^{n-1} \right\rbrace.
\end{equation}

By classical arguments, one can show that the infimum in \eqref{Rel} is achieved and in what follows we will denote by $u_p\in W^{1,p}(\Omega)$ the eigenfunction corresponding to the first eigenvalue $\Lambda_p$.

In this paper, we firstly prove that
\begin{equation}
 \lim_{p\rightarrow+\infty} (\Lambda_p)^{1/p} =\label{infi}
  \Lambda_\infty=: \inf_{\substack{w\in W^{1,\infty}(\Omega) \\ \norma{w}_{L^\infty(\Omega)}=1}}
  \max \left\lbrace \norma{\nabla w}_{L^\infty(\Omega)}, \beta\norma{w}_{L^\infty(\partial\Omega)} \right\rbrace,
\end{equation}
and we give a geometric characterization of this quantity, precisely:

\begin{equation}\label{lambinf}
   \Lambda_\infty = \frac{1}{1/\beta +R_\Omega},
\end{equation}
where $R_\Omega$ denotes the inradius of $\Omega$, i.e. the radius of the l argest ball contained in $\Omega$. Thereafter, we prove that $\Lambda_\infty$ is the first eigenvalue of the infinite Laplacian, in the sense that equation \eqref{ciao2} admits non-trivial solutions only if $\Lambda \geq \Lambda_\infty$.

Similar results, in the case of Dirichlet and Neumann boundary conditions were obtained in \cite{JLM,JL,BDM,EKNT, RS2}.

More specifically, in \cite{JLM,JL}, Juutinen, Lindqvist and Manfredi have studied the Dirichlet case as $p \rightarrow +\infty$. They provided a complete characterization of the limiting solutions in terms of geometric quantities. 
Indeed, the first eigenvalue of the $p$-Laplace operator $\lbrace\lambda_p^D\rbrace$ happens to satisfy
$$
\lim_{p\to \infty} \left(\lambda_p^D\right)^{{1}/{p}} = \lambda^D_\infty:= \frac{1}{R_\Omega}.
$$
The related eigenfunctions $v^D_p$ also converge (up to a subsequence) to some Lipschitz function $v^D_\infty$.
Most important, the Authors show that there exists a natural viscosity formulation of the eigenvalue problem for the $\infty$-Laplacian, for which $\lambda^D_\infty$ and $v^D_\infty$ turn out to be the first eigenvalue and first eigenfunction, respectively. 

The Neumann case seems to be more subtle. It was investigated in \cite{EKNT, RS2} and similarly to the Dirichlet case, the Authors established that the first non-trivial eigenvalues of the $p$-Laplacian $\lbrace\lambda^N_p \rbrace$ satisfy
$$\lim_{p\to \infty}\left(\lambda^N_p\right)^{1/p} = \lambda_\infty^N := \frac{2}{\diam(\Omega)},$$
where $\diam(\Omega)$ is the intrinsic diameter of $\Omega$, i.e. the supremum of the geodetic distance between two points of $\Omega$.

However, while both first eigenvalues and first eigenfunctions converge (as $p\to\infty$) and are solutions to some appropriate eigenvalue problem for the $\infty$-Laplacian, in \cite{EKNT}, the Authors are able to prove that they actually converge to the \emph{first} eigenvalue e \emph{first} eigenfunction only if the domain $\Omega$ is convex. Whether or not the same holds true in the general case, it is still an open problem.

In the second part of the paper, we focus our attention on the study of the limit of the $p$-Poisson equation with Robin boundary conditions:

\begin{equation}
\label{p_originale2}
\begin{cases}
-\Delta_p v= f & \text{ in } \Omega \\
\abs{\nabla v}^{p-2} \displaystyle{\frac{\partial v}{\partial \nu}} + \beta^p \abs{v}^{p-2}v =0 & \text{ on } \partial \Omega,
\end{cases}
\end{equation}
when $f \in L^\infty(\Omega)$ is a non-negative function.

We prove that there exists (up to a subsequence) a limiting solution $v_\infty$ as $p\to\infty$
and we establish conditions on $f$ which are equivalent to the uniqueness of $v_\infty$.

The $\infty$-Poisson problem for Dirichlet boundary conditions
was already studied in \cite{BDM} by Bhattacharya, DiBenedetto and Manfredi, while, to the best of our knowledge, similar results have not been addressed in the case of Neumann boundary conditions.

 \section{Notations and Preliminaries}
\label{section_notion}
Throughout this article, $|\cdot|$ will denote the Euclidean norm in $\mathbb{R}^n$, and $\mathcal{H}^k(\cdot)$, for $k\in [0,n)$, will denote the $k-$dimensional Hausdorff measure in $\mathbb{R}^n$.

We denote by $d(x,\partial \Omega)$ the distance function from the boundary, defined as
\begin{equation}
  d(x,\partial \Omega)= \inf_{y \in \partial\Omega} \abs{x-y},
\end{equation}
for an exhaustive discussion about this function and its properties see \cite{GT}. Moreover, we recall that the inradius $R_\Omega$ of $\Omega$ is 
\begin{equation}\label{inradiuss}
  R_\Omega=\sup_{x\in\Omega}\inf_{y\in\partial\Omega}|x-y|= \norma{d(\cdot,\partial \Omega)}_{L^\infty(\Omega)}.
  \end{equation}

The following lemma makes us understand why \eqref{infi} can be seen as a limit problem of \eqref{Rel}.

\begin{lemma}
\label{lemma1}
Given $f, g \in W^{1,\infty}(\Omega)$, then 
$$
\lim_{p \to \infty} \left( \int_\Omega \abs{f}^p + \int_\Omega \abs{g}^p\right)^{1/p} = \max\left\lbrace \norma{f}_\infty , \norma{g}_\infty \right\rbrace.
$$
\end{lemma}
\begin{proof} 

The proof of this lemma can be found in \cite{RS}.
\end{proof}

\subsection{Viscosity solutions}
Before going on, we recall  the definition of viscosity solutions to a boundary value problem, see \cite{CIL} for more details.
\begin{definizione}
  We consider the following boundary value problem 
    \begin{equation}
    \label{1}
        \begin{cases}
            F(x, u, \nabla u, D^2u) =0 &\text{ in }\Omega, \\
            B(x,u,\nabla u)=0 &\text{ on }\partial \Omega,
        \end{cases}
    \end{equation}
    where $F:\R^n\times\R\times\R^n\times\R^{n\times n}\to \R$ and $B:\R^n\times \R\times \R^n\to\R$ are two continuous functions.
    \begin{description}
        \item[Viscosity supersolution] A lower semi-continuous function $u$ is a viscosity supersolution to \eqref{1} if, whenever we fix $x_0\in \overline{\Omega}$, for every $\phi \in C^2(\overline{\Omega})$ such that $u(x_0)=\phi (x_0)$ and $x_0$ is a strict minimum in $\Omega$ for $u - \phi$, then
        \begin{itemize}
            \item if $x_0 \in \Omega$, the following holds
            $$
            F\left(x_0,\phi(x_0), \nabla\phi(x_0) ,D^2\phi(x_0)\right) \geq 0
            $$
            \item if $x_0 \in \partial \Omega$, the following holds
            $$
            \max\Set{F\left(x_0,\phi(x_0),\nabla\phi(x_0) ,D^2\phi(x_0)\right), B\left(x_0, \phi(x_0), \nabla\phi(x_0)\right)}\geq 0
            $$
        \end{itemize}
        \item[Viscosity subsolution] An upper semi-continuous function $u$ is a viscosity subsolution to \eqref{1} if,  whenever we fix $x_0\in \overline{\Omega}$, for every $\phi \in C^2(\overline{\Omega})$ such that $u(x_0)=\phi (x_0)$ and  $x_0$ is a strict maximum in $\Omega$ for $u - \phi$, then
        \begin{itemize}
            \item if $x_0 \in \Omega$, the following holds
            $$
            F\left(x_0,\phi(x_0), \nabla\phi(x_0) ,D^2\phi(x_0)\right) \leq 0
            $$
            \item if $x_0 \in \partial \Omega$, the following holds
            $$
            \min\Set{F\left(x_0,\phi(x_0), \nabla\phi(x_0) ,D^2\phi(x_0)\right), B\left(x_0, \phi(x_0), \nabla\phi(x_0)\right)}\leq 0
            $$
        \end{itemize}
         \item[Viscosity solution] A continuous function $u$ is a viscosity solution to \eqref{1} if it is both a super and subsolution.
    \end{description}
\end{definizione}
\begin{oss}
The condition $u-\phi$ has a strict maximum or minimum can be relaxed: it is sufficient to ask that $u-\phi$ has a local maximum or minimum in a ball $B_R(x_0)$ for some positive $R$. 
\end{oss}

\section{The $\infty$-eigenvalue problem}
\label{eigenva}
Let us start this section observing that Lemma \ref{lemma1} brings to the following estimate
\begin{equation}
\label{lpminlinf}
    \limsup_{p\to \infty} \Lambda_p^{1/p} \leq\Lambda_\infty.
\end{equation}

We can say something more

\begin{lemma}
\label{th1}
Let $\Set{\Lambda_p}_{p>1}$ be the sequence of the first eigenvalues of the $p$-Laplacian operator with Robin boundary condition. Then,
\begin{equation}
    \lim_{p\to \infty} \left(\Lambda_p\right)^{\frac{1}{p}}=\Lambda_\infty,
\end{equation}
where $\Lambda_\infty$ is defined in \eqref{infi}.

Moreover, if $\Set{u_p}_{p>1}$ is the sequence of eigenfunctions associated to $\{\Lambda_p\}_{p>1}$, then there exists a function $u_\infty\in W^{1,\infty}(\Omega)$ such that, up to a subsequence,
\begin{align*}
   & u_p \to u_{\infty} & \text{ uniformly in } \, \Omega \\
   & \nabla u_p\to \nabla u_\infty & \text{ weakly in } \, L^q(\Omega), \forall q.
   \end{align*}
\end{lemma}
\begin{proof}
As a consequence of \eqref{lpminlinf}, the sequence $\Set{u_p}_{p>1} $ of eigenfunctions associated to $\Lambda_p$ is uniformly bounded in $W^{1,q}(\Omega)$: indeed, if $q <p$, by H\"older inequality,
\begin{gather}
    \label{unifbound} \norma{\nabla u_p}_{L^q(\Omega)} \leq \norma{\nabla u_p}_{L^p(\Omega)} \abs{\Omega}^{\frac{1}{q} - \frac{1}{p}} \leq\Lambda_p^{1/p} \abs{\Omega}^{\frac{1}{q} - \frac{1}{p}} \leq  C,\\
    \label{unifbound2} \norma{ u_p}_{L^q(\Omega)} \leq \norma{ u_p}_{L^p(\Omega)} \abs{\Omega}^{\frac{1}{q} - \frac{1}{p}} \leq \abs{\Omega}^{\frac{1}{q} - \frac{1}{p}} \leq  C,
\end{gather}
where the constant $C$ is independent of $p$.

By a classical argument of diagonalization, see for instance \cite{BDM}, we can extract a subsequence $u_{p_j}$ such that
\begin{gather*}
     u_{p_j} \to u_\infty \,  \text{ uniformly}  \implies \Vert u_{p_j}\Vert_{L^{p_j}} \to \norma{u_{\infty}}_{L^{\infty}}, \\
    \nabla u_{p_j} \to \nabla u_\infty \, \text{ weakly in } \, L^q(\Omega), \, \forall q >1.
\end{gather*} 

Moreover, from \eqref{unifbound} and \eqref{unifbound2}, the following inequality holds
\begin{align*}
    \frac{\norma{\nabla u_\infty}_{L^q(\Omega)}}{\norma{u_\infty}_{L^q(\Omega)}}&\le \liminf_{p \to \infty} \frac{\norma{\nabla u_p}_{L^q(\Omega)}}{\norma{u_p}_{L^q(\Omega)}}\le
     \liminf_{p \to \infty} \frac{\norma{\nabla u_p}_{L^p(\Omega)}}{\norma{u_p}_{L^q(\Omega)}}  \abs{\Omega}^{\frac{1}{q}-\frac{1}{p}}\\ &\le  \frac{\abs{\Omega}^{\frac{1}{q}}}{\norma{u_\infty}_{L^q(\Omega)}} \liminf_{p \to \infty} \left( \Lambda_p\right)^{\frac{1}{p}}.
\end{align*}
Letting $q\to\infty$ we obtain $$\norma{\nabla u_\infty}_{L^\infty(\Omega)}\le \liminf_{p \to \infty} \left( \Lambda_p\right)^{\frac{1}{p}} $$

Similarly
$$
\beta \norma{ u_p}_{L^q(\partial\Omega)} \leq \beta \norma{\ u_p}_{L^p(\partial\Omega)} \abs{\partial\Omega}^{\frac{1}{q} - \frac{1}{p}} \leq\Lambda_p^{1/p} \abs{\partial\Omega}^{\frac{1}{q} - \frac{1}{p}}\leq  C,
$$
gives us
$$
\beta \norma{u_\infty}_{L^\infty(\partial\Omega)}\le \liminf_{p \to \infty} \left( \Lambda_p\right)^{\frac{1}{p}},
$$
  hence
 $$
\Lambda_\infty \leq \liminf_{p \to \infty} \Lambda_p^{1/p}.
$$
\end{proof}

Now we want to show that the limit $u_\infty$ solves \eqref{ciao2} in viscosity sense, but before we need the following proposition

\begin{prop}
\label{teorema1.3}
A continuous weak solution $u$ to \eqref{autoval} is a viscosity solution to \eqref{autoval}.
\end{prop}
\begin{proof}
The proof is similar to the one in \cite{JLM, EKNT} for the $p$-Laplacian with other boundary conditions. 

We only write explicitly the proof that $u_\infty$ satisfies the boundary conditions in the viscosity sense.

Let $u$  be a continuous weak solution to \eqref{autoval}, let $x_0 \in \partial \Omega$ and let us consider a function $\phi $ such that $\phi(x_0)=u(x_0)$ and such that $u-\phi$ has a strict minimum at $x_0$.
Then \begin{equation}
\begin{multlined}
     \max \Big\{-\abs{\nabla \phi(x_0)}^{p-2} \Delta \phi(x_0)- (p-2) \abs{\nabla \phi(x_0)}^{p-4}\Delta_\infty \phi (x_0)- \Lambda_p \abs{\phi(x_0)}^{p-2}\phi(x_0),\\ \abs{\nabla \phi(x_0)}^{p-2} \frac{\partial \phi(x_0)}{ \partial \nu} + \beta^p \abs{\phi(x_0)}^{p-2} \phi(x_0) \Big\} \geq0.
\end{multlined}
\end{equation}
Assume by contradiction that both terms are negative.
If we choose $r$ sufficiently small, in $\overline{\Omega}\cap B_r(x_0 )$, we have
$$
     -\abs{\nabla \phi(x)}^{p-2} \Delta \phi(x)- (p-2) \abs{\nabla \phi(x)}^{p-4}\Delta_\infty \phi (x)- \Lambda_p \abs{\phi(x)}^{p-2}\phi(x) < 0
     $$
and, in $\partial\Omega \cap B_r(x_0 )$,
$$
\abs{\nabla \psi(x)}^{p-2} \frac{\partial \psi(x)}{ \partial \nu} + \beta^p \abs{\psi(x)}^{p-2} \psi(x) <0, \quad\text{where } \psi = \phi + \frac{m}{2}.
$$

Then
\begin{equation*}
    \begin{multlined}
           \int_{\Set{\psi > u} \cap B_r(x_0 )} \abs{\nabla\psi}^{p-2} \nabla \psi \nabla (\psi - u)\, dx \\<\Lambda_p \int_{\Set{\psi > u} \cap B_r(x_0 )} \abs{\phi}^{p-2} \phi (\psi - u) \, dx - \beta^p \int_{\partial \Omega \cap B_r(x_0 ) \cap  \Set{\psi >u} } \abs{\psi}^{p-2} \psi (\psi - u) \, d\mathcal{H}^{n-1},
    \end{multlined}
\end{equation*}
using the definition of weak solution, we have
\begin{equation*}
\begin{aligned}
  C(N,p) \int_{\Set{\psi > u} \cap B_r(x_0 )} &\abs{\nabla \psi - \nabla u}^p \, dx \\ &\leq \int_{\Set{\psi > u} \cap B_r(x_0 )} \left\langle \abs{\nabla\psi}^{p-2} \nabla \psi - \abs{\nabla u}^{p-2} \nabla u, \nabla(\psi - u)\right\rangle \, dx\\  
  &< \Lambda_p \int_{\Set{\psi > u} \cap B_r(x_0 )} \left(\abs{\phi}^{p-2} \phi-\abs{u}^{p-2} u\right) (\psi - u) \, dx \\  &- \beta^p \int_{\partial \Omega \cap B_r(x_0 ) \cap  \Set{\psi >u}} \left(\abs{\psi}^{p-2}\psi- \abs{u}^{p-2} u\right) (\psi - u)\, d\mathcal{H}^{n-1} <0
\end{aligned}
\end{equation*}
which gives a contradiction.
\end{proof}

Now we can prove the following
\begin{teorema}
Let $u_\infty$ be  the function given in Theorem \ref{th1}. Then $u_\infty$ is a viscosity solution to
\begin{equation}
    \label{ciao}
    \begin{cases}
    \min\Set{\abs{\nabla u} - \Lambda_\infty u, - \Delta_\infty u}=0 &\text{ in } \Omega,\\
    -\min\Set{\abs{\nabla u} - \beta u, -\displaystyle{\frac{\partial u}{\partial \nu}} }=0 &\text{ on } \partial \Omega.
\end{cases}
\end{equation}

\end{teorema}
\begin{proof}We divide the proof in two steps.

\textbf{Step 1} $u_\infty$ is a viscosity supersolution.\\ Let $x_0\in \Omega$  and let $\phi\in C^2(\Omega)$ be such that $u_\infty-\phi$ has a strict minimum in $x_0$. We want to show
$$
\min\Set{\abs{\nabla \phi(x_0)} - \Lambda_\infty \phi(x_0), - \Delta_\infty \phi(x_0)}\geq 0 
$$ 
Notice that $u_p- \phi$ has a minimum in $x_p$ and $x_p \to x_0$. If we set $\phi_p(x) =\phi(x) + c_p$ with $c_p= u_p(x_p) - \phi(x_p)\to 0$ when $p$ goes to infinity,  we have that $u_p(x_p)= \phi_p(x_p) $ and $ u_p- \phi_p$ has a minimum in $x_p$, so Proposition \ref{teorema1.3} implies
\begin{equation}
    \label{7}
    -\abs{\nabla \phi_p(x_p)}^{p-2} \Delta \phi_p(x_p)- (p-2) \abs{\nabla \phi_p(x_p)}^{p-4}\Delta_\infty \phi (x_p)- \Lambda_p \abs{\phi_p(x_p)}^{p-2}\phi_p(x_p) \geq0.
\end{equation}
Now dividing by $(p-2) \abs{\nabla \phi_p(x_p)}^{p-4}$, we obtain
\begin{equation}
\label{f1}
   - \Delta_\infty \phi_p(x_p) - \frac{\abs{\nabla \phi_p(x_p)}^{2 } \Delta \phi_p(x_p)}{ p-2} \geq \frac{ \abs{\nabla \phi_p(x_p)}^4}{(p-2) \phi_p(x_p)}	\left(\frac{\Lambda_p^{1/p} \phi_p(x_p)}{ \abs{\nabla \phi_p(x_p)}}\right)^p
\end{equation}

  This gives us
$
\abs{\nabla \phi(x_0)}-\Lambda_\infty \phi(x_0)\geq 0 $ since, otherwise, the right-hand side of \eqref{f1} would go to infinity, in contradiction with the fact that $\phi \in C^2(\Omega)$. Moreover 
$- \Delta_\infty \phi(x_0)\geq 0
$, just taking the limit.

  Then, $\min \Set{\abs{\nabla \phi(x_0)}-\Lambda_\infty \phi(x_0), - \Delta_\infty \phi(x_0)}\ge 0$ and $u_\infty$ is a viscosity supersolution.

Let us fix  $x_0 \in \partial \Omega$, $\phi\in C^2(\overline{\Omega})$ such that  $u-\phi$ has a strict minimum in $x_0$, our aim is to prove that 
$$
\max\Set{\min\Set{\abs{\nabla \phi(x_0)} - \Lambda_\infty \phi(x_0), - \Delta_\infty \phi(x_0)}, -\min\Set{\abs{\nabla \phi(x_0)}- \beta \phi(x_0), - \frac{\partial \phi}{\partial \nu} (x_0)}} \geq 0
$$
If for infinitely many $x_p\in \Omega$ \eqref{7} holds true, then we get
$$
\min\Set{\abs{\nabla \phi(x_0)} - \Lambda_\infty \phi(x_0), - \Delta_\infty \phi(x_0)} \geq 0.
$$
If for infinitely many $p$, $x_p \in \partial \Omega$ the following holds true 
$$
\abs{\nabla \phi_p(x_p)}^{p-2} \frac{\partial \phi_p(x_p)}{ \partial \nu} + \beta^p \abs{\phi_p(x_p)}^{p-2} \phi_p(x_p)  \geq 0,
$$
then
$$
\abs{\nabla \phi_p(x_p)}^{p-2}\left( - \frac{\partial \phi_p(x_p)}{ \partial \nu}\right)\leq \beta^p \abs{\phi_p(x_p)}^{p-2} \phi_p(x_p).
$$

Only two cases can occur:
\begin{itemize}
    \item $\displaystyle{- \frac{\partial \phi}{\partial \nu}(x_0) \leq 0}$;
    \item $\displaystyle{- \frac{\partial \phi}{\partial \nu}(x_0) > 0}$, then letting $p$ to infinity in the following
   
$$
\left(\abs{\nabla \phi_p(x_p)}^{p-2}\left( - \frac{\partial \phi_p(x_p)}{ \partial \nu}\right) \right)^{1/p}\leq \left(\beta^p \abs{\phi_p(x_p)}^{p-2} \phi_p(x_p)\right)^{1/p} 
$$

we get $\displaystyle{\abs{\nabla \phi(x_0)} \leq \beta \phi (x_0)}$.
\end{itemize}
That is
$$
-\min\Set{\abs{\nabla \phi(x_0)}- \beta \phi(x_0), - \frac{\partial \phi}{\partial \nu} (x_0)} \geq 0.
$$

\textbf{Step 2}  $u_\infty$ is a viscosity subsolution.

Let us fix $x_0\in\Omega$, $\phi\in C^2(\Omega)$ such that $u_\infty-\phi$ has a strict maximum. We want to prove that $$\min\left\lbrace\abs{\nabla \phi(x_0)}-\Lambda_\infty \phi(x_0), - \Delta_\infty \phi(x_0)\right\rbrace\le 0,$$ so it is enough to prove that only one of the two terms in the bracket is non positive.

For instance, assume  that $-\Delta_\infty \phi (x_0)>0$, we can argue as in \eqref{7}, but now, all the inequality involving the second order differential operator are reversed and we get
\begin{equation*}
\Lambda_p \phi_p^{p-1}(x_p) \ge (p-2) \abs{\nabla \phi_p(x_p)}^{p-4}
\left[-\frac{\abs{\nabla \phi_p(x_p)}^2\Delta\phi_p(x_p)}{p-2}
- \Delta_\infty \phi_p(x_p)\right].
\end{equation*}
As $-\Delta_\infty \phi(x_0)>0$, the term in the big parenthesis is non-negative, we can erase everything to the power $1/p$, obtaining
$$\Lambda_\infty \phi(x_0)\ge \abs{\nabla \phi(x_0)},$$
which shows that $u_\infty$ is a viscosity subsolution to \eqref{ciao}.

Similar arguments to step 1 give us the boundary conditions for viscosity subsolution.
\end{proof}

  We are also able to give a geometric characterization of $\Lambda_\infty$. 
\begin{lemma}
\label{charaterization}
Let $\Lambda_\infty$ be the quantity defined in \eqref{infi} and let $R_\Omega$ be the inradius of $\Omega$. Then
\begin{equation*}
    \Lambda_\infty = \min_{x_0 \in \Omega} \frac{1}{\frac{1}{\beta } + d (x_0,\partial \Omega )}= \frac{1}{\frac{1}{\beta} + R_\Omega}.
\end{equation*}
\end{lemma}
\begin{proof}
The function $\frac{1}{\beta} + d(x, \partial \Omega) \in W^{1, \infty}(\Omega)$, moreover
$$
\norma{\nabla \left({1}/{\beta} + d(x, \partial \Omega)\right)}_{L^\infty(\Omega)}=1\quad \text{ and } \quad \beta \norma{ {1}/{\beta} + d(x, \partial \Omega)}_{L^\infty(\partial \Omega)}=1.
$$
Then
$$
\Lambda_\infty \leq \min_{x_0 \in \Omega} \frac{1}{\frac{1}{\beta} + d(x_0, \partial \Omega)}.
$$
In order to prove the reverse inequality, we take $w \in W^{1,\infty}(\Omega)$ such that $\norma{w}_{L^\infty(\Omega)}=1$. 

The following facts can occur
\begin{description}
\item[Case 1]If $\beta \norma{w}_{L^\infty(\partial\Omega)} \leq \norma{\nabla w}_{L^\infty(\Omega)} $, then 
$$
\max \left\lbrace \norma{\nabla w}_{L^\infty(\Omega)}, \beta\norma{w}_{L^\infty(\partial\Omega)} \right\rbrace= \norma{\nabla w}_{L^\infty(\Omega)}.
$$
  We choose $x \in \Omega$ and $y$ equal to the point on the boundary such that $\abs{x-y}=d(x,\partial\Omega)$. So, we have
\begin{equation*}
    \begin{aligned}
     \abs{w(x)}&\leq \abs{w(x) -w(y)} +\abs{w(y)} \\& \leq \norma{\nabla w}_{L^\infty(\Omega)}\abs{x-y} + \norma{w}_{L^\infty(\partial\Omega)} \\&\le \norma{\nabla w}_{L^\infty(\Omega)}d(x,\partial \Omega)+ \frac{1}{\beta } \norma{\nabla w}_{L^\infty(\Omega)}\\ &= \norma{\nabla w}_{L^\infty(\Omega)} \left( \frac{1}{\beta } +d(x,\partial \Omega) \right) \\ &\le \norma{\nabla w}_{L^\infty(\Omega)} \norma{ {1}/{\beta } +d(x,\partial \Omega)}_{L^\infty(\Omega)},
    \end{aligned}
\end{equation*}
and then
$$
\frac{\norma{\nabla w}_{L^\infty(\Omega)}}{\norma{w}_{L^\infty(\Omega)}}\geq\frac{1}{ \norma{ {1}/{\beta } +d(x,\partial \Omega)}_{L^\infty(\Omega)}}
$$
\item[Case 2]If $\beta \norma{w}_{L^\infty(\partial\Omega)} > \norma{\nabla w}_{L^\infty(\Omega)} $, then 
$$
\max \left\lbrace \norma{\nabla w}_{L^\infty(\Omega)}, \beta\norma{w}_{L^\infty(\partial\Omega)} \right\rbrace= \beta\norma{w}_{L^\infty(\partial\Omega)}.
$$
With the same choice of $x$ and $y$, we have
\begin{equation*}
    \begin{aligned}
     \abs{w(x)}&\leq \abs{w(x) -w(y)} +\abs{w(y)} \\& \leq \norma{\nabla w}_{L^\infty(\Omega)}\abs{x-y} + \norma{w}_{L^\infty(\partial\Omega)} \\&\le \beta \norma{w}_{L^\infty(\partial\Omega)}d(x,\partial \Omega)+  \norma{w}_{L^\infty(\partial\Omega)} \\ &= \beta \norma{w}_{L^\infty(\partial\Omega)}\left(d(x,\partial \Omega)+ \frac{1}{\beta } \right) \\ &\le \beta \norma{w}_{L^\infty(\partial\Omega)} \norma{ {1}/{\beta } +d(x,\partial \Omega)}_{L^\infty(\Omega)}.
    \end{aligned}
\end{equation*}
Hence,
$$
\frac{\beta \norma{w}_{L^\infty(\partial\Omega)}}{\norma{w}_{L^\infty(\Omega)}}\geq \frac{1}{ \norma{ {1}/{\beta } +d(x,\partial \Omega)}_{L^\infty(\Omega)}}.
$$
\end{description}
Finally we get $\forall w \in W^{1,\infty}(\Omega):\norma{w}_{L^\infty(\Omega)}=1,$
$$
\max \left\lbrace \norma{\nabla w}_{L^\infty(\Omega)}, \beta\norma{w}_{L^\infty(\partial\Omega)} \right\rbrace \ge  \frac{1}{ \norma{ {1}/{\beta } +d(x,\partial \Omega)}_{L^\infty(\Omega)}} $$
and then the desired inequality
$$
\Lambda_\infty\geq \min_{x_0 \in \Omega} \frac{1}{\frac{1}{\beta} + d(x_0, \partial \Omega)}.
$$

\end{proof}

\begin{teorema}
Let $\Lambda_\infty$ be the quantity defined in \eqref{infi}. Then$$\Lambda_\infty(\Omega)\ge \Lambda_\infty(\Omega^\sharp),$$
where $\Omega^\sharp$ is the ball centered at the origin with the same measure of $\Omega$.
\end{teorema}
\begin{proof}
The Faber-Krahn inequality for the first eigenvalue of the $p$-Laplacian with Robin boundary condition (for instance see \cite{BD}) states that
$$\Lambda_p(\Omega)\ge \Lambda_p(\Omega^\sharp).$$

Letting $p$ go to infinity, we have
$$\Lambda_\infty(\Omega)\ge \Lambda_\infty(\Omega^\sharp).$$

This can follow also from the geometric characterization in Lemma \ref{charaterization}
$$\Lambda_\infty= \frac{1}{\frac{1}{\beta}+ R_\Omega}$$
as the ball maximizes the inradius among sets of given volume.  
\end{proof}
\begin{oss}
One can easily prove that the function $\frac{1}{\beta}+d(x,\partial\Omega)$ is an eigenfunction  if the domain $\Omega= B_R(x_0)$. This is not true if $\Omega$ is a square: see for instance \cite{JLM}.
\end{oss}

\subsection{The first Robin $\infty$-eigenvalue}

Now we want to show that $\Lambda_\infty$ is the first eigenvalue of \eqref{ciao2}, that is the smallest $\Lambda$ such that
\begin{equation*}
    \begin{cases}
    \min\Set{\abs{\nabla u} - \Lambda u, - \Delta_\infty u}=0 &\text{ in } \Omega,\\
    -\min\Set{\abs{\nabla u} - \beta u, -\displaystyle{\frac{\partial u}{\partial \nu}} }=0 &\text{ on } \partial \Omega
\end{cases}
\end{equation*}
admits a non-trivial solution.

\begin{teorema}
\label{firsteig}
Let $\Omega$ be a bounded and open set of class $C^2$ in $\R^n$. If for some $\Lambda$, problem \eqref{ciao2} admits a non-trivial eigenfunction $u$, then $\Lambda \geq \Lambda_\infty$.
\end{teorema}
\begin{proof}
Let $\Lambda$ be an eigenvalue to \eqref{ciao2}, let $u$ be a corresponding  eigenfunction. We normalize it in this way
$$
\max_{x \in \Omega} u(x)= \frac{1}{\Lambda}.
$$
Then $u$ is viscosity subsolution to
$$
\min\Set{\abs{\nabla u} - 1, - \Delta_\infty u}=0 \text{ in } \Omega.
$$
For every $\varepsilon >0 $ and $\gamma >0$, let us consider the function $$g_{\varepsilon, \gamma} = \frac{1}{\beta} + (1+ \varepsilon) d(x , \partial \Omega ) - \gamma d(x , \partial \Omega)^2.$$

It is well known (see \cite{GT}) that in a small tubular neighbourhood $\Gamma_\mu$ of $\partial\Omega$, the boundary $\partial\Omega$ and the distance function $d(x,\partial\Omega)$ share the same regularity: so both $d(x,\partial\Omega)$ and $g_{\varepsilon,\gamma}$ are $C^2(\Gamma_\mu)$.


  Moreover, by a direct computation, one can check that if $$\gamma < \frac{\varepsilon}{2 R_\Omega}, $$  
then $g_{\varepsilon, \gamma}$ is a viscosity supersolution to
$$
\min\Set{\abs{\nabla g_{\varepsilon, \gamma}} - 1, - \Delta_\infty g_{\varepsilon, \gamma}}=0 \text{ in } \Omega.
$$

  Hence, Theorem $2.1$ in \cite{Jen} ensures that
$$
m_\varepsilon = \inf_{x \in \Omega} (g_{\varepsilon , \gamma}(x)-u(x)) = \inf_{x \in \partial \Omega} (g_{\varepsilon , \gamma}(x)-u(x)).
$$
Assume by contradiction that $m_\varepsilon < -\frac{\varepsilon}{\beta}$, and set $v = g_{\varepsilon , \gamma} - m_\varepsilon$. We observe that $v \geq u$ in $\Omega$ and $v(x_0)=u(x_0)$, where $x_0$ is the point which achieves the minimum on the boundary, so we can use it as test function in the definition of viscosity subsolution for $u$. 

Assuming $\gamma < \frac{\varepsilon}{2 R_\Omega}$,  we obtain

\begin{gather*}
    \nabla v(x ) = \left[ 1+ \varepsilon -2 \gamma d(x,\partial \Omega )\right] \nabla  d(x,\partial \Omega ) \\
\abs{\nabla v(x_0)}= 1+ \varepsilon -2 \gamma d(x_0,\partial \Omega )>1\\
-\frac{\partial v}{\partial \nu}(x_0) = - \left[ 1+ \varepsilon -2 \gamma d(x_0,\partial \Omega )\right] \nabla  d(x_0,\partial \Omega )\cdot \nu >0\\
- \Delta_\infty v (x_0)= 2\gamma \left[1+ \varepsilon - 2\gamma d(x_0, \partial \Omega)\right]^2\abs{\nabla d(x_0, \partial \Omega)}^4 >0.
\end{gather*}
The fact that $m_\varepsilon < -\frac{\varepsilon}{\beta}$ implies
\begin{equation*}
    \abs{\nabla v(x_0)} - \beta v(x_0) = \varepsilon+ \beta  m_\varepsilon <0.
\end{equation*}
  Therefore 
$$
-\min \left\lbrace \abs{\nabla v} - \beta v, -\frac{\partial v}{\partial \nu} \right\rbrace >0 \quad \text{ and } \quad \min \left\lbrace \abs{\nabla v} -  1, -\Delta_\infty v \right\rbrace >0
$$
against the fact that
$$
\min \left\lbrace \min \left\lbrace \abs{\nabla v} -  1, -\Delta_\infty v \right\rbrace ,-\min \left\lbrace \abs{\nabla v} - \beta v, -\frac{\partial v}{\partial \nu} \right\rbrace \right\rbrace\leq0.
$$

So we have 
$$
g_{\varepsilon , \gamma}(x) - u(x) \geq m_\varepsilon \geq - \frac{\varepsilon}{\beta},
$$
letting $\varepsilon $ and $\gamma $ go to zero, it follows
$$
\frac{1}{\beta} + d(x,\partial \Omega ) \geq u(x) \qquad \forall x \in \Omega.
$$

  Hence
$$
\frac{1}{\Lambda_\infty}= \max_{x \in \Omega}\left(\frac{1}{\beta} + d(x,\partial \Omega )\right) \geq \max_{x \in \Omega} u(x)= \frac{1}{\Lambda},
$$
which concludes the proof.
\end{proof}

\section{The $p$-Poisson equation}
\label{poisso}
Let $f$ be a function in $ L^{\infty}(\Omega)$ and let $\beta >0$. We consider the following $p$-Poisson equation with Robin boundary conditions
\begin{equation}
\label{p_originale}
\begin{cases}
-\Delta_p v= f & \text{ in } \Omega \\
\abs{\nabla v}^{p-2} \displaystyle{\frac{\partial v}{\partial \nu}} + \beta^p  \abs{v}^{p-2}v =0  & \text{ on } \partial \Omega.
\end{cases}
\end{equation}
A function $v_p$ is a weak solution to \eqref{p_originale} if it satisfies
\begin{equation}
\label{weak-f}
    \int _\Omega \abs{\nabla v_p}^{p-2} \nabla v_p \nabla \varphi \, dx+ \beta^p \int_{\partial\Omega}\abs{v_p}^{p-2}v_p\varphi \, d \mathcal{H}^{n-1} = \int_{\Omega} f \varphi \, dx, \quad \forall \varphi \in W^{1,p}(\Omega).
\end{equation}

It is well known that the solution to this equation is the unique minimum of the functional

\begin{equation}
    \label{funzi}
    J_p(\varphi)=\frac{1}{p} \int_{\Omega} \abs{\nabla \varphi}^p \, dx +\frac{\beta^p}{p}\int_{\partial\Omega} \abs{\varphi}^p \, d\mathcal{H}^{n-1}(x) - \int_{\Omega} f\varphi \, dx.
\end{equation}
Indeed, it is possible to prove the existence and the uniqueness of the minimum of the functional thanks to the so-called direct method of calculus of variation, see for instance \cite{Dac, G, Lind,AGM}.

  If we let \emph{formally} $p$ go to $\infty$ in \eqref{funzi}, we obtain the functional

\begin{equation}
    \label{limi}
     \varphi \longrightarrow \min \int_\Omega -f \varphi\, dx \quad \varphi \in W^{1,\infty}(\Omega).
\end{equation}
The limit procedure imposes two extra constraints to \eqref{limi}, namely

$$\norma{\nabla \varphi}_{\infty}\le 1, \quad \beta\norma{ \varphi}_{L^\infty(\partial\Omega)}\le 1.$$ 

The following proposition holds true.
\begin{prop}
\label{uinff}
Let $v_p$ be the solution to \eqref{p_originale}. Then there exists a subsequence $\Set{v_{p_j}}_{j}$ such that
$$ v_{p_j} \to v_\infty     \,  \text{ uniformly,} \qquad
    \nabla v_{p_j} \to \nabla v_\infty \, \text{ weakly in } \, L^m(\Omega), \, \forall m >1.
$$
Moreover
$$
\norma{\nabla v_\infty}_\infty \leq 1 \qquad \beta \norma{ v_\infty}_{L^\infty(\partial\Omega)} \leq 1.
$$
\end{prop}
\begin{proof}
  Choosing $\varphi= v_p$ in \eqref{weak-f}, we have
$$
\int _\Omega \abs{\nabla v_p}^{p}  + \beta^p \int_{\partial \Omega} v_p^{p}= \int_{\Omega} f v_p,
$$
and Young inequality gives
\begin{equation}
\label{yeye}
\int _\Omega \abs{\nabla v_p}^{p}  + \beta^p \int_{\partial \Omega} v_p^{p} - \frac{1}{\varepsilon_p^p p} \int_\Omega v_p^p\leq \frac{\varepsilon_p^{p'}}{p'} \int_\Omega f^{p'}.
\end{equation}
Taking into account 
\eqref{Rel}, we get
\begin{equation*}
   \begin{multlined}
\left( 1- \frac{1}{p \Lambda_p \varepsilon_p^p }\right) \left[ \int_\Omega \abs{\nabla v_p}^{p}  + \beta^p \int_{\partial \Omega} v_p^{p}\right] \leq  \qquad \qquad \qquad\\
    \qquad\leq\int_\Omega \abs{\nabla v_p}^{p}  + \beta^p \int_{\partial \Omega} v_p^{p} - \frac{1}{\varepsilon_p^p p} \int_\Omega v_p^p\leq \frac{\varepsilon_p^{p'}}{p'} \int_\Omega f^{p'}.
\end{multlined} 
\end{equation*}

Choosing $\varepsilon_p $ such that $\displaystyle{1- \frac{1}{p \Lambda_p \varepsilon_p^p }=\frac{1}{2}}$, we have
\begin{equation}
\label{stima}
    \int_\Omega \abs{\nabla v_p}^p+ \beta^p\int_{\partial \Omega} v_p^p\le 2 \frac{\varepsilon_p^{p'}}{p'} \int_\Omega f^{p'} \leq C \int_\Omega f^{p'}
\end{equation}
 where the constant  $C$ is independent of $p$, thanks to Lemma \ref{lemma1}.

Hence 
\begin{gather*}
    \left(\int _\Omega \abs{\nabla v_p}^{p} \right)^{1/p} \leq\left( C \int_\Omega f^{p'}\right)^{1/p} \leq \left(C \abs{\Omega}\, \norma{f}^{p'}_\infty \right)^{1/p},\\
\left(\beta^p\int _{\partial \Omega}  v_p^{p} \right)^{1/p} \leq\left( C \int_\Omega f^{p'}\right)^{1/p} \leq \left(C \abs{\Omega}\, \norma{f}^{p'}_\infty \right)^{1/p}.
\end{gather*}

Analogously
\begin{equation}
    \label{stima 2}
   \left( \int_\Omega  v_p^p\right)^{1/p}\leq C\left( \int_\Omega f^{p'}\right)^{1/p}.
\end{equation}

If $p>m$, H\"older inequality gives
\begin{equation}
    \label{gras}
    \begin{gathered}
    \left(\int _\Omega \abs{\nabla v_p}^{m} \right)^{1/m}  \leq \left(\int _\Omega \abs{\nabla v_p}^{p} \right)^{1/p} \abs{\Omega}^{1/m - 1/p }\leq \left( C  \norma{f}^{p'}_\infty \right)^{1/p}\abs{\Omega}^{1/m },\\
    \beta \left(\int _{\partial \Omega} v_p^{m} \right)^{1/m}  \leq \left(\beta^p\int _{\partial \Omega}  v_p^{p} \right)^{1/p} \abs{\partial \Omega}^{1/m - 1/p }\leq \left( C  \norma{f}^{p'}_\infty \right)^{1/p}\abs{\partial \Omega}^{1/m },\\
    \left(\int_\Omega v_p^m\right)^{1/m} \leq  \left(\int_\Omega v_p^p\right)^{1/p} \abs{\Omega}^{1/m-1/p} \leq C\left( \norma{f}_\infty^{p'}\right)^{1/p}\abs{\Omega}^{1/m},
    \end{gathered}
\end{equation}
Then there exists  $v_{p_j}$ such that
\begin{align*}
     v_{p_j} \to v_\infty \,  &\text{ uniformly,} \qquad
    \nabla v_{p_j} \to \nabla v_\infty \, \text{ weakly in } \, L^m(\Omega), \, \forall m >1,
\end{align*}
moreover
\begin{gather*}
    \norma{\nabla v_\infty}_m \leq \liminf_{j \to \infty} \norma{\nabla v_{p_j}}_m \leq \lim_{j \to \infty} \left( C  \norma{f}^{p_j'}_\infty \right)^{1/{p_j}}\abs{\Omega}^{1/m }= \abs{\Omega}^{1/m } \\
    \beta \norma{v_\infty}_{L^m(\partial\Omega)} =\beta \lim_{j\to\infty} \norma{ v_{p_j}}_{L^m(\partial\Omega)} \leq \lim_{j\to\infty} \left( C  \norma{f}^{p_j'}_\infty \right)^{1/p_j}\abs{\partial\Omega}^{1/m }= \abs{\partial\Omega}^{1/m }
\end{gather*}

and then $$  \norma{\nabla v_\infty}_{L^\infty(\Omega)} \le 1 \quad \quad \beta \norma{v_\infty}_{L^\infty(\partial\Omega)}\le 1. 
$$
\end{proof}

We are also able  to link the so obtained function $v_\infty$ with the functional \eqref{limi}, indeed
\begin{teorema}\label{th42}
The functional 
\begin{equation}
\label{minimo}
    J_\infty(\varphi)= - \int_\Omega f \varphi \qquad \varphi \in W^{1,\infty}(\Omega)
\end{equation}
admits at least one minimum $\overline{\varphi}$ satisfying $\norma{ \nabla \overline{\varphi}}_{L^{\infty}( \Omega)}  \leq 1$ and $\beta \norma{\overline{\varphi}}_{L^{\infty}(\partial \Omega)} \leq 1$.

Moreover, if $v_\infty$ is a limit of a subsequence of $\{v_p\}$, then $v_\infty$ is a minimizer of \eqref{minimo}.
\end{teorema}
\begin{proof}
Let $v_\infty$ a limit of a subsequence of $\{v_p\}$ and let us assume it is not a minimum of $J_\infty$. This implies that $\exists \varphi \in W^{1,\infty}(\Omega)$ such that
$$
-\int_\Omega f\varphi < - \int_\Omega f v_\infty.
$$
We want to show that there exists a function $\phi$ and an exponent $p$, such that $J_p(\phi) < J_p(v_p)$, which contradicts the minimality of $v_p$.

  First of all, we recall that exists a sequence $v_{p_i}\rightharpoonup v_\infty$ in $W^{1,m}(\Omega)\, \, \forall m$. Then 
$$
\int_\Omega f v_{p_i} \to \int_\Omega f v_\infty
$$
and so there exists $\overline{i}$ for which we still have 
\begin{equation}
\label{assurd}
    -\int_\Omega f\varphi < -\int_\Omega f v_{p_i} \qquad \forall i > \overline{i}.
\end{equation}

  Two cases can occur:
\begin{description}
\item[case 1] $\exists i > \overline{i} $
$$
\int_\Omega \abs{\nabla \varphi}^{p_i} + \beta^{p_i} \int_{\partial \Omega } \abs{\varphi}^{p_i}  \leq \int_\Omega \abs{\nabla v_{p_i}}^{p_i} + \beta^{p_i} \int_{\partial \Omega }  v_{p_i}^{p_i}   
$$
  Then
\begin{equation*}
    \begin{split}
    J_{p_i}(\varphi) &= \frac{1}{p}\int_\Omega \abs{\nabla \varphi}^{p_i} + \frac{\beta^{p_i}}{p} \int_{\partial \Omega } \abs{\varphi}^{p_i}  -\int_\Omega f\varphi  \qquad \\
    &< \frac{1}{p}\int_\Omega \abs{\nabla v_{p_i}}^{p_i} + \frac{\beta^{p_i}}{p} \int_{\partial \Omega }  v_{p_i}^{p_i} -\int_\Omega f v_{p_i}  = J_{p_i}(v_{p_i}),
    \end{split}
\end{equation*}
  which is a contradiction.
\item[case 2] $\forall i >\overline{i}$
$$
\int_\Omega \abs{\nabla \varphi}^{p_i} + \beta^{p_i} \int_{\partial \Omega } \abs{\varphi}^{p_i}  > \int_\Omega \abs{\nabla v_{p_i}}^{p_i} + \beta^{p_i} \int_{\partial \Omega }  v_{p_i}^{p_i}  .
$$
Considering $\phi = \alpha \varphi$ with $\alpha \in (0,1)$:
$$
-\int_\Omega f\phi= -  \alpha \int_\Omega f \varphi < -\int_\Omega f v_{p_i} < 0 \qquad \forall i > \overline{i},
$$
we have 
$$
\int_\Omega \abs{\nabla \phi}^{p_i} + \beta^{p_i} \int_{\partial \Omega } \abs{\phi}^{p_i}  = \alpha^{p_i}\left[ \int_\Omega \abs{\nabla \varphi}^{p_i} + \beta^{p_i} \int_{\partial \Omega }  \abs{\varphi}^{p_i}\right].
$$
Moreover
\begin{equation*}
    \begin{cases}
    \displaystyle{ M \leq \int_\Omega f v_{p_i} = \int_\Omega \abs{\nabla v_{p_i}}^{p_i} + \beta^{p_i} \int_{\partial \Omega }  v_{p_i}^{p_i}}\\
     \displaystyle{\int_\Omega \abs{\nabla \varphi}^{p_i} + \beta^{p_i} \int_{\partial \Omega } \abs{\varphi}^{p_i} \leq  \abs{\Omega} \norma{ \nabla \varphi}_{L^{\infty}( \Omega)}^{p_i} + \beta^{p_i} \abs{\partial \Omega}\norma{\varphi}_{L^{\infty}(\partial \Omega)}^{p_i}   \leq \abs{\Omega} + \abs{\partial \Omega}}.
    \end{cases}
\end{equation*}
We now choose $p_i$:
$$
0 \xleftarrow{i \to \infty} \alpha^{p_i} \leq \frac{M}{\abs{\Omega} +  \abs{\partial \Omega}} \leq \frac{\int_\Omega \abs{\nabla v_{p_i}}^{p_i} + \beta^{p_i} \int_{\partial \Omega }  v_{p_i}^{p_i}}{\int_\Omega \abs{\nabla \varphi}^{p_i} + \beta^{p_i} \int_{\partial \Omega } \abs{\varphi}^{p_i}}
$$
obtaining
$$
\int_\Omega \abs{\nabla \phi}^{p_i} + \beta^{p_i} \int_{\partial \Omega } \abs{\phi}^{p_i}  \leq \int_\Omega \abs{\nabla v_{p_i}}^{p_i} + \beta^{p_i} \int_{\partial \Omega }  v_{p_i}^{p_i}. \qedhere
$$
\end{description}
\end{proof}

\begin{prop}
\label{prop1}
Let $v_p$ be the solution to \eqref{p_originale} and let $v_\infty$ be any  limit of a subsequence of $\Set{v_p}_{p>1}$. Then
\begin{equation}
    \label{minoreq}
    v_\infty(x)\le \frac{1}{\beta}+ d(x,\partial\Omega).
\end{equation}
\end{prop} 
\begin{proof}
We notice that
$$\abs{v_\infty(x)-v_\infty(y)}\le \abs{x-y}$$
as we have proven that $\norma{\nabla v_\infty}_\infty\le 1$. This holds true for every $x,y$ in $\Omega$. In particular, we can choose $y$ equal to the point on the boundary which realizes $\abs{x-y}=d(x,\partial\Omega)$. So, we have

$$ v_\infty(x)\le v_\infty(y)+d(x, \partial\Omega)\le \frac{1}{\beta}+d(x, \partial\Omega),$$
as $v_\infty$ also satisfies $\beta \norma{v_\infty}_{L^\infty(\partial\Omega)}\le1$.
\end{proof}
\begin{oss}
\label{remark4.1}
We explicitly observe that the estimate $\varphi(x)\le \frac{1}{\beta}+d(x,\partial\Omega)$ holds for every admissible function $\varphi$.
\end{oss}

\begin{prop}
\label{prop2}
Assume $f>0$ in $\Omega$. Then the sequence of solution to \eqref{p_originale} converges strongly in $W^{1,m}(\Omega)$, for all $m>1$, to
$$\overline{v}_\infty(x)=\frac{1}{\beta}+d(x,\partial\Omega).$$
\end{prop}
\begin{proof}
Let $v_\infty $ be any limit of a subsequence $\Set{v_{p_j}}\subset \Set{v_p}$. 
Theorem \ref{th42} gives that $v_\infty$ is a minimum of the functional $J_\infty$, in the class $\{\varphi\in  W^{1,\infty}(\Omega):\, \norma{\nabla\varphi}_\infty\le 1,\, \beta \norma{\varphi}_{L^\infty(\partial\Omega)}\le1\}$.

The function $\frac{1}{\beta}+d(x,\partial\Omega)$ is a competitor and then
\begin{equation}
\label{lolo}
    \int_\Omega f \left(v_\infty- \frac{1}{\beta}- d(x, \partial\Omega)\right)\ge0.
\end{equation}
Then \eqref{minoreq} gives  $v_\infty(x)=\frac{1}{\beta}+d(x,\partial\Omega)$.

  Since every subsequence of $\Set{v_p}$ has a subsequence converging to $\frac{1}{\beta}+d(x,\partial\Omega)$ weakly in $W^{1,m}(\Omega)$, the whole sequence $\Set{v_p}$ converges to $\frac{1}{\beta}+d(x,\partial\Omega)$ weakly in $W^{1,m}(\Omega)$, and in particular, in $C^{\alpha}(\overline{\Omega})$ and its gradient weakly in $L^m(\Omega)$.

It remains to prove the strong convergence in $W^{1,m}(\Omega)$.
 
Clarkson's inequality  implies for $p,q >m $
\begin{equation*}
    \begin{multlined}
    \int_\Omega \frac{\abs{\nabla v_p + \nabla v_q}^m}{2^m} + \int_\Omega \frac{\abs{\nabla v_p - \nabla v_q}^m}{2^m} \leq \frac{1}{2}\int_\Omega \abs{\nabla v_p}^m +\frac{1}{2}\int_\Omega \abs{\nabla v_q}^m 
    \end{multlined}
\end{equation*}
From \eqref{gras} we deduce
$$
\lim_{p\to\infty} \int_\Omega \abs{\nabla v_p}^m \leq \abs{\Omega},
$$
then semicontinuity of $L^m $-norm gives
$$
\limsup_{p,q}\int_\Omega \frac{\abs{\nabla v_p + \nabla v_q}^m}{2^m} \leq  \abs{\Omega}= \int_\Omega \abs{\nabla d(x, \partial \Omega)}^m \leq \liminf_{p,q}\int_\Omega \frac{\abs{\nabla v_p + \nabla v_q}^m}{2^m} 
$$
and then
$$
\limsup_{p,q}\int_\Omega \frac{\abs{\nabla v_p - \nabla v_q}^m}{2^m} =0.
$$

\end{proof}

\begin{oss}
\label{oss4}
If $\text{Supp}f\subset\Omega$, then we can deduce that $v_\infty(x)=\frac{1}{\beta}+ d(x,\partial\Omega)$ for all $x\in \text{Supp}f$, while in $\Omega\setminus \text{Supp}f$ inequality \eqref{minoreq} can be strict. 
\end{oss}

  \begin{definizione}
    We denote by $\mathcal{R}$ the set of discontinuity of the function $\nabla d(x, \partial \Omega)$. This set consists of points $x \in \Omega$ for which  $d(x, \partial\Omega)$ is achieved by more than one point $y$ on the boundary. 
\end{definizione}

Then it holds true the following
  
\begin{teorema} Let $f$ be a non-negative function in $\Omega$, then function $\displaystyle{\overline{v}_\infty(x)= \frac{1}{\beta}+ d(x, \partial\Omega)}$ is the unique extremal function of \eqref{minimo} if and only if $ \mathcal{R}\subset \text{Supp}f$.
\end{teorema}

\begin{proof}
Suppose that $\mathcal{R}\subset \text{Supp}f$ and let $w$ be a minimum of \eqref{minimo}in the class $\{\varphi\in  W^{1,\infty}(\Omega):\, \norma{\nabla\varphi}_\infty\le 1,\, \beta \norma{\varphi}_{L^\infty(\partial\Omega)}\le1\}$. By Remark \ref{remark4.1} we have
$$
w(x) \leq \frac{1}{\beta}+ d(x, \partial\Omega) \qquad \forall x \in \Omega,
$$
and arguing as in Remark \ref{oss4} we have 
$$
w(x) = \frac{1}{\beta}+ d(x, \partial\Omega) \qquad \forall x \in \overline{\text{supp}f}.
$$
Assume by contradiction that there exists $x \in (\overline{\text{Supp}f})^c$ such that
$$
w(x) < \frac{1}{\beta}+ d(x, \partial\Omega),
$$
setting $\eta=\nabla d(x,\partial\Omega)$, we choose the smallest $t$ such that $y=x+t\eta$ belongs to $\partial (\text{Supp}f)$ (Lemma \ref{lemma_seg} will justify this choice), 
thus
\begin{equation*}
    \begin{aligned}
        \norma{\nabla w}_{L^\infty(\Omega)} \abs{y-x} &\geq w(y) - w(x) \\&>d(y, \partial\Omega) - d(x, \partial\Omega)= \nabla d(\xi, \partial\Omega)\cdot (y-x) = \abs{y-x}
    \end{aligned}
\end{equation*}
where the last equality follows from Lemma \ref{lemma_seg}.

  So we have $\norma{\nabla w}_{L^\infty(\Omega)}>1$ that is a contradiction.

Assume now that $\displaystyle{w(x)= \frac{1}{\beta}+ d(x, \partial\Omega)}$ is the unique extremal of \eqref{minimo}, and assume by contradiction that $ \mathcal{R}\not\subset \Supp f$. Then, thanks to the Aronsson theorem (see \cite{Aron}), we can construct a function $\varphi$ different from $w$ which coincide with $w$ in $\Supp f$ and such that it is admissible for \eqref{limi}. Then $\varphi$ is a minimum too.\\
This contradicts the fact that $w$ is the unique minimum of \eqref{limi}, so $\mathcal{R}\subset \text{Supp}f$.

\end{proof}

We have to prove Lemma \ref{lemma_seg} to complete the proof.
\begin{lemma}
\label{lemma_seg}
Let $x \in \Omega\setminus \mathcal{R}$ and set $\eta=\nabla(d(x,\partial \Omega))$. Let us consider $y_t= x +t \eta$, then there exists $T$ such that $y_T\in \mathcal{R}$ and $y_t\notin \mathcal{R}$ for all $t<T$. Moreover, 
$$
\nabla d(x+t \eta,\partial \Omega)= \eta \qquad \forall t \in [0, T).
$$
\end{lemma}
\begin{proof}
Consider the following Cauchy problem
\begin{equation}
    \begin{cases}
        \dot{\gamma}(t)= \nabla d(\gamma(t), \partial \Omega) ,\\
        \gamma(0)=x
    \end{cases}
\end{equation}
in a maximal interval $ [0,T)$. We have that
\begin{itemize}
\item $\displaystyle{L(\gamma) = \int_0^T \abs{\dot{\gamma}(t) }\, ds=T}$;
    \item $\displaystyle{\frac{d}{dt}d(\gamma(t),\partial\Omega)= \nabla d(\gamma(t), \partial \Omega) \dot{\gamma}(t)=1,}$
    then 
    $$
     T=\int_0^T \frac{d}{dt}d(\gamma(t),\partial\Omega) \, dt = d(\gamma(T), \partial \Omega) - d(x, \partial \Omega).
    $$
\end{itemize}
These considerations give us the following:
\begin{itemize}
    \item $T < \infty$, otherwise $d$ is unbounded, and this is a contradiction as $\Omega$ is bounded;
    \item $\gamma(T) \in \mathcal{R}$, otherwise one can extend the solution for $t>T$, in contradiction with the fact that $[0,T)$ is the maximal interval.
\end{itemize}
In the end, if $y = \gamma(T)$, we have
$$
d(y, \partial \Omega) = d(x, \partial \Omega) + T= d(x, \partial \Omega) + L(\gamma),
$$
 $L(\gamma) \geq \abs{y-x} $, and they are equal if and only if $\gamma $ is a segment.

  If $L(\gamma) > \abs{y-x} $ then $$d(y, \partial \Omega) = d(x, \partial \Omega) + L(\gamma) > d(x, \partial \Omega) + \abs{y-x} \geq \abs{y-z}$$ with $z\in \partial \Omega$ such that $d(x ,\partial\Omega) =\abs{x-z}$, and this is a contradiction, because $d(y,\partial\Omega)$ is the infimum.
Then $L(\gamma) = \abs{y-x} $ and, remembering the fact $\dot{\gamma}(t)= \nabla d(\gamma(t), \partial \Omega)$, whose norm is 1, then $\gamma$ is a segment and
$$
\nabla d(\gamma(t), \partial \Omega)= \eta \qquad \forall t \in [0,T).
$$
This conclude the proof.
\end{proof}
\subsection{The limit PDE}
We have proved that any limit $v_\infty$ of subsequence $\Set{v_p}$ is a minimum of the functional \eqref{minimo} defined in $W^{1,\infty}(\Omega)$. Now we want to understand if such limits are solutions to a certain PDE, which, in some sense, 
is the Euler-Lagrange equation of the functional \eqref{minimo}. 

\begin{prop}
Let $f\in L^\infty(\Omega)\cap C(\overline{\Omega})$ be a non-negative function. Then any limit $v_\infty$ of a subsequence $\Set{v_p}$ satisfies
\begin{equation}
\label{ali1}  
    \abs{\nabla v_\infty}\leq 1  \quad \qquad \qquad \text{in the viscosity sense}.
\end{equation}
\end{prop}
\begin{proof}
The proof follows the techniques contained in \cite{BDM, JLM,YY,LW}.

  If $x_0\in \Omega$, let $\varphi\in C^2(\Omega)$ be such that $v_\infty-\varphi$ has a local maximum at $x_0$
$$(v_\infty-\varphi)(x_0)\ge(v_\infty-\varphi)(x), \, \, \forall x\in B_R(x_0),$$
then
\begin{equation}
    \label{visc1}
    \abs{\nabla \varphi(x_0)}\le 1. 
\end{equation}
Indeed, let $$C=\sup_{q> 1} \max{\Set{\norma{v_q}_{L^\infty\left( B_R(x_0)\right)}; \norma{\varphi}_{L^\infty\left( B_R(x_0)\right)}}},$$
we consider the following sequence
$$ f_q(x)= v_q(x)- \varphi(x)-k\abs{x-x_0}^a, \quad k=\frac{4C}{R^a}, \quad a>2.$$
So $f_q(x)\le -3C$ for $x\in\partial B_R(x_0)$, and $f_q(x_0)\ge -2C$. Then $f_q(x)$ attains its maximum at some point $x_q$ in the interior of $B_R(x_0)$.
and it holds
$$\nabla v_q(x_q)=\nabla \varphi(x_q)+ k a (x_q-x_0)\abs{x_q-x_0}^{a-2}.$$
Moreover, the sequence $\Set{x_q}$ must converge to $x_0$. 

 Assume by contradiction that $\abs{\nabla \varphi(x_0)}>1$, so there exists $\delta\in(0,1)$ such that $\abs{\nabla \varphi(x_0)}\ge 1+\delta$. Choosing $\overline{q}$ large enough, we have

$$\abs{\nabla v_q(x_q)}>1+\frac{\delta}{2} - ka \abs{x_q-x_0}^{a-1}\ge 1+\frac{\delta}{4}, \quad \forall q\ge \overline{q}.$$

  By Lemma 1.1 of Part III in \cite{BDM}, 
\begin{equation}
    \abs{\nabla v_q (x)} \leq \left(\frac{\gamma}{R^n}\right)^{\frac{1}{q}}\left(\int_{B_{\frac{R}{2}(x_0)}}\left(1+ \abs{\nabla v_q}\right)^q\, dx\right)^{\frac{1}{q}}, \qquad \forall x \in B_\frac{R}{2}(x_0)
\end{equation}
where $\gamma $ is a constant independent of $q$.


  For $q$ sufficiently large, this contradict $\abs{\nabla v_q(x_q)}>1+\frac{\delta}{4}$. 
\end{proof}
\begin{prop}
\label{Teorema5.3}Let $f\in L^\infty(\Omega)\cap C(\Omega)$ be a non-negative function, then a continuous weak solution to \eqref{p_originale} is a viscosity solution.
\end{prop}
\begin{proof}
The proof follows the same techniques of Proposition \ref{teorema1.3}.
\end{proof}
\begin{teorema}
\label{Teorema5.2}
Let $f\in L^\infty(\Omega)\cap C(\Omega)$  be a non-negative function. Then any $v_\infty$ satisfies
\begin{align}
\label{ali}
    &\abs{\nabla v_\infty}=1  \quad \qquad \qquad \,\text{on} \, \Set{f>0}&& \text{in the viscosity sense}\\
    \label{ali2}
    &-\Delta_\infty v_\infty=0 \quad \qquad \qquad \text{on}\, (\overline{\Set{f>0}})^c && \text{in the viscosity sense}
\end{align}
\end{teorema}

\begin{proof}
We first prove \eqref{ali}. Let $x_0\in \Omega \cap \Set{f >0}$  and let $\varphi\in C^2(\Omega)$ such that $v_\infty-\varphi$ has a strict minimum in $x_0$. We want to show
$$
\abs{\nabla \varphi(x_0)}\geq 1 
$$ 
Let us denote by $x_p$ the minimum of $v_p- \varphi$, we have that  $x_p \to x_0$, then $x_p \in B_R(x_0) \subset \Set{f >0}$  for $p $ large enough. Setting $\varphi_p(x) =\varphi(x) + c_p$, then $c_p= v_p(x_p) - \varphi(x_p)\to 0$ when $p$ goes to infinity.  We notice that $v_p(x_p)= \varphi_p(x_p) $ and $ v_p- \varphi_p$ has a minimum in $x_p$, so by Proposition \ref{Teorema5.3},
$$
 -\abs{\nabla \varphi_p(x_p)}^{p-2} \Delta \varphi_p(x_p)- (p-2) \abs{\nabla \varphi_p(x_p)}^{p-4}\Delta_\infty \varphi (x_p) \geq f(x_p) >0.
$$
Dividing by $(p-2) \abs{\nabla \varphi_p(x_p)}^{p-4}$, we obtain
\begin{equation}
    \label{sus}
    - \Delta_\infty \varphi_p(x_p) - \frac{\abs{\nabla \varphi_p(x_p)}^{2 } \Delta \varphi_p(x_p)}{ p-2} \geq \frac{ f(x_p)}{(p-2)\abs{\nabla \varphi_p(x_p)}^{p-4}}
\end{equation}

  This gives us
$
\abs{\nabla \varphi(x_0)}\geq 1 $, otherwise the right-hand side would go to infinity, in contradiction with the fact that $\varphi \in C^2(\Omega)$.

 We stress that, if we let $p\to \infty $ in \eqref{sus}, we get $-\Delta_\infty \varphi(x_0)\ge 0$, so $v_\infty$ is a supersolution to $-\Delta_\infty u=0$.

 Now, we only have to prove that $v_\infty$ is a viscosity solution to $\Delta_\infty \varphi=0$ in $(\overline{\Set{f>0}})^c$.

 If we fix $x_0\in(\overline{\Set{f>0}})^c$, and $\varphi \in C^2(\Omega)$ such that $v_\infty - \varphi$ has a strict maximum in $x_0$, we can choose $R$ such that $B_R(x_0)\subset(\overline{\Set{f>0}})^c$. The function $v_p - \varphi $ has a maximum $x_p \to x_0$, and $x_p \in B_R(x_0)$ for $p$ large enough. 
 
 The definition of viscosity subsolution implies
$$
 -\abs{\nabla \varphi_p(x_p)}^{p-2} \Delta \varphi_p(x_p)- (p-2) \abs{\nabla \varphi_p(x_p)}^{p-4}\Delta_\infty \varphi (x_p) \leq f(x_p)=0.
$$

 Without loss of generality, we may assume $\abs{\nabla \varphi(x_0)}\neq 0$. Dividing both side of the last equation by $(p-2) \abs{\nabla \varphi_p(x_p)}^{p-4}$, we obtain
$$
-\Delta_\infty \varphi_p(x_p) \leq \frac{\abs{\nabla \varphi_p(x_p)}^{2} \Delta \varphi_p(x_p)}{p-2}.
$$

 Letting $p\to \infty$, we get
 
 $$-\Delta_\infty \varphi(x_0)\le 0.$$

 Analogously if $\varphi \in C^2(\Omega)$ is such that $v_\infty - \varphi$ has a minimum at $x_0$, a symmetric argument shows that $-\Delta_\infty \varphi (x_0)\geq 0$.
\end{proof}

\subsection*{Some examples}

\begin{esempio}
\label{ex1}
We consider the case $\Omega= B_1(0)$ and $f=1$.
The uniqueness of the solution and the invariance under rotation of the $p$-Laplacian ensure that the solution $v$ is radially symmetric
and 
$$
r^{n-1} \Delta_p v_p= \frac{d}{dr} \left( r^{n-1} \left\lvert v_p'\right\rvert^{p-2} v_p' \right).
$$
Setting $\alpha=1/(p-1)$, we have
$$
v_p(x)= - \frac{p-1}{ n^\alpha p} \abs{x}^{\frac{p}{p-1}} + \frac{1}{(n\beta^p)^\alpha} + \frac{p-1}{n^\alpha p}
$$
Then 
$$
v_\infty (x) = -\abs{x} +\frac{1}{\beta} +1= \frac{1}{\beta}+ d(x, \partial \Omega)
$$
\end{esempio}
\begin{esempio}
\label{ex2}
We fix $0<\varepsilon<1$ and we consider
$$
\left.
f= 
\right.
\begin{cases}
1 &\text{ if } x \in B_\varepsilon(0) \\
0 &\text{ if } x \in B_1(0)\setminus B_\varepsilon(0).
\end{cases}
$$

  In this case, $v_p$ is radially symmetric, and

$$
\left.
v_p= 
\right.
\begin{cases}
\frac{p-1}{n^\alpha p} \left( \varepsilon^{\frac{p}{p-1}} -\abs{x}^{\frac{p}{p-1}} \right) + \frac{\varepsilon^{n\alpha}(p-1)}{n^\alpha (p-n)} \left( 1- \varepsilon^{\frac{p-n}{p-1}} \right)+ \frac{\varepsilon^{n\alpha}}{(n\beta^p)^\alpha } &\text{ if } x \in B_\varepsilon(0) \\
\frac{\varepsilon^{n\alpha}(p-1)}{n^\alpha (p-n)} \left( 1-\abs{x}^{\frac{p-n}{p-1}}\right) + \frac{\varepsilon^{n\alpha}}{(n\beta^p)^\alpha } &\text{ if } x \in B_1(0)\setminus B_\varepsilon(0)
\end{cases}
$$

  Letting $p$ go to infinity, we obtain
$$
v_\infty= \frac{1}{\beta}+ d(x,\partial \Omega).
$$
\end{esempio}

\addcontentsline{toc}{chapter}{
Bibliografia}

\addcontentsline{toc}{chapter}{ Bibliografia}
\bibliographystyle{alpha}
\bibliography{biblio}
\nocite{Barl,GMP}

\Addresses

\end{document}